\documentclass[12pt, reqno]{amsart}
\usepackage{amsmath, amsthm, amscd, amsfonts, amssymb, graphics, color}
\usepackage[bookmarksnumbered, plainpages] {hyperref}
\theoremstyle{proclaim}
\newtheorem{theorem}{Theorem}[section]
\newtheorem{lemma}[theorem]{Lemma}

\newtheorem{proposition}[theorem]{Proposition}

\theoremstyle{definition}
\newtheorem{remark}[theorem]{Remark}
\newtheorem{definition}[theorem]{Definition}

\numberwithin{equation}{section}
\begin{document}

\title[Simple  reduced $L^p$ operator crossed products]{Simple reduced $L^p$ operator crossed products with unique trace}
\author[Shirin Hejazian,\; Sanaz Pooya ]{Shirin Hejazian,\; Sanaz Pooya}

\address{Shirin Hejazian, Department of Pure Mathematics,
Ferdowsi University of Mashhad,  Mashhad 91775, Iran;
}
\email{hejazian@um.ac.ir}
\address{Sanaz Pooya, Department of Pure Mathematics, Ferdowsi University of Mashhad,
Mashhad 91775, Iran}
\email{pooya.snz@gmail.com}

\begin{abstract}
Given $p \in (1, \infty)$, let $G$ be a countable Powers group, and let  $(G, A, \alpha)$ be
a separable nondegenerately representable isometric $G$-$L^p$ operator algebra. We show that if $A$ is $G$-simple and unital then the reduced $L^p$ operator crossed product of $A$ by $G$, $F^p_{\mathrm{r}}(G, A, \alpha)$, is simple. This generalizes  special cases of some results due to de la Harpe and Skanadalis in the $C^*$-algebra context. We will also show that the result is not true for $p=1$. Moreover, we prove that traces  on $F^p_{\mathrm{r}}(G, A, \alpha)$ are in  natural bijection  with $G$-invariant traces on $A$ via the standard conditional expectation. As a consequence, for a countable powers group $G$  the reduced $L^p$ operator group algebra, $F^{p}_{\mathrm{r}}(G)$,  is simple and has a unique normalized trace.
\end{abstract}

\maketitle
% All section titles should be typed in capitals (less the math commands).
\section{INTRODUCTION}
For a discrete group $G$, its regular representation in $L^2(G)$ generates a $C^*$-algebra $C_{\mathrm{r}}^*(G)$ with a faithful trace. Such objects are interesting both in analytical and group theoretical contexts, a fact that became apparent from a result of Powers \cite{Po} which says that the reduced group $C^*$-algebra of a free group with two generators is simple and has a unique trace.

A group $G$ is called \emph{$C^*$-simple} if it is infinite and if its reduced group $C^*$-algebra has no nontrivial two-sided ideals.
Since the announcement of Powers' result in 1975, the class of
  $C^*$-simple groups and in general simple $C^*$-algebras has been considerably enlarged. For more recent examples see \cite{AM, BH, H1, dlhp, OO}. Indeed many authors applied his distinguished approach to some other groups which sometimes lead to defining new classes of $C^*$-simple groups.
One of those interesting classes is the class of \emph{Powers groups} defined in \cite{PLH}, see Definition \ref{pw} below.
These groups enjoy  both combinatorial and geometrical properties. As a first example one can think of nonabelian free groups. During recent years some authors by doing modifications in the definition of a Powers group have  introduced  new examples of $C^*$-simple groups and  study properties of the latter, cf. \cite{BN, B,  dlhp, Pr}.

  In \cite{HS}, de la Harpe and Skandalis  among other results proved  that the reduced $C^*$-crossed product, $C^*_{\mathrm{r}}(G, A,\alpha)$, by a Powers group $G$, on a unital $G$-simple $C^*$-algebra $A$, is simple and its traces are characterized in terms of traces on $A$.

Since the theory of crossed products has been developed, crossed products of other algebras than $C^*$-algebras and von Neumann algebras have received very little attention. But very recent efforts suggest that there is an interesting theory behind these. Indeed, in a new approach, Dirksen, de Jeu and Wortel in \cite{M1} defined crossed products of Banach algebras and Phillips in \cite{PhC} studied crossed products of a specific class of Banach algebras, the so-called $L^p$ operator algebras.
In fact,  Phillips along his way to compute the $K$-theory of an $L^p$ version of Cuntz algebras,  introduced crossed products of operator algebras on $\sigma$-finite $L^p$ spaces by isometric actions of locally compact groups, for $p \in [1, \infty)$.
In his very recent works on $L^p$ operator algebras, among many different results, he has  introduced some simple $L^p$ operator algebras.
The reader may refer to \cite{PhCr, PhC,  PhI}  for  details.

 Being interested in investigating simple $L^p$ operator algebras of crossed product type we are going to generalize  the main results of  \cite{HS} in the $L^p$ level for countable Powers groups.

 This paper is arranged as follows. Section 2 contains some preliminaries which are needed in the sequel. In Section 3, we state our main results.
 Here we should emphasize  that,  because of some technical requirements, in the definition of full and reduced $L^p$ operator crossed products \cite[Definition 3.3]{PhCr},  $G$ is assumed to be a second countable locally compact group. Hence in order to make our discrete groups fit in with this framework we need to consider countable Powers groups. In \cite{HS} the Powers groups are not assumed to be countable and there is no condition on the $C^*$-algebra other than it be unital. Here we will assume that the Powers group $G$ is countable and the unital $L^p$ operator algebra $A$ is  separable.  So the main results of this paper, in particular for $p=2$, generalize special cases of the main results in \cite{HS}.   We prove that the reduced $L^p$ operator crossed product, $F^p_{\mathrm{r}}(G, A, \alpha)$, is simple whenever  $p \in (1, \infty)$, $G$ is a countable Powers group,  $(G, A, \alpha)$ is a separable nondegenerately representable isometric $G$-$L^p$ operator algebra, and $A$ is  $G$-simple and unital. Furthermore, we show that  each trace on $F^p_{\mathrm{r}}(G, A, \alpha)$ is of the form $\sigma \circ E$, where $\sigma$  is a $G$-invariant trace on $A$ and $E$ is the standard conditional expectation from $F^p_{\mathrm{r}}(G, A, \alpha)$ to $A$, conversely,  $\sigma\circ E$ is always a trace on the reduced crossed product. As a consequence for $p\in(1, \infty)$, $F^{p}_{\mathrm r}(G)$, the $L^p$ operator group algebra of a countable Powers group $G$ is simple with a unique normalized trace. From this we can easily deduce  $L^p$ versions of some results by Powers \cite{Po} and Paschke and Salinas \cite{PS}.  It also turns out that for any countable discrete  group $G$ neither  $F^1_{\mathrm{r}}(G)$ nor $F^1(G) $ is simple, as one can see from \cite[Proposition 3.14]{PhCr}.

 Here in the lack of convenient geometric properties of Hilbert spaces, in particular that every closed subspace has  an orthogonal complement, we prove our key results by using the duality between $L^p$ and $L^q$ spaces when $p, q$ are conjugate exponents.

\section{PRELIMINARIES}

In this section we recall some basic definitions, examples and results mainly from \cite{PhCr}, in order to make this article self-contained.

Let $(X, \mathcal B, \mu)$ be a measure space. For  $p\in [1, \infty]$, we denote by $B(L^p(X, \mu))$ the Banach algebra of all bounded linear operators on $L^p(X, \mu)$.

 An \emph{$L^p$ operator algebra} is defined to be
 a Banach algebra $A$  which is isometrically
isomorphic to a norm closed subalgebra of $B(L^p(X, \mu))$ for some measure space $(X, \mathcal B, \mu)$ and  $p\in [1, \infty]$.
For  $p = 2$,  an $L^2$ operator algebra $A$ is isometrically
isomorphic to a norm closed (not necessarily self-adjoint) subalgebra of the
bounded operators on some Hilbert space.

Clearly, for $p \in [1,\infty]$ and  measure space $(X, \mathcal B, \mu)$, the algebra $B(L^p(X, \mu))$ is an $L^p$ operator algebra. Also if
$X$ is a locally compact Hausdorff space, then $C_0(X)$, with its supremum norm, is an $L^p$ operator algebra for all  $p \in [1,\infty]$, cf. \cite[Example 1.13]{PhCr}.
\begin{definition}\label{1}(\cite[Definition 1.17] {PhCr})
Let $p\in[1,\infty]$, and let $A$ be an $L^p$ operator algebra.\begin{enumerate}
\item [(i)] Let $(X, \mathcal B, \mu)$ be a measure space.  A \emph{representation} of $A$ (on $L^p(X, \mu)$) is a continuous homomorphism $\pi\colon A \to B(L^p(X, \mu))$. If $\| \pi(a)\|\leq \|a\|$ (resp.\ $\| \pi(a)\|=\|a\|$) for all $a\in A$, then $\pi$ is called \emph{contractive} (resp.\ \emph{isometric}).
\item [(ii)] Let $p \neq \infty$. A representation $\pi\colon A \to B(L^p(X, \mu))$ is called \textit{separable} if $L^p(X, \mu)$ is separable.
\item [(iii)] A representation $\pi$ is called \emph{$\sigma$-finite} if $\mu$ is $\sigma$-finite.
\item [(iv)] A representation $\pi$ is called \emph{nondegenerate} if
$$\pi(A)(L^p(X, \mu)) = \text{span} {(\{\pi(a)\xi \colon a\in A \,\,\text{and} \,\, \xi \in L^p(X, \mu)\})}$$
is dense in $L^p(X, \mu)$.
\item [(v)] We say that  $A$ is   \emph{separably} (\emph{nondegenerately})  \textit{representable} whenever
it has a separable (nondegenerate)  isometric representation, and \emph{nondegenerately $\sigma$-finitely representable} if it has a nondegenerate $\sigma$-finite isometric
representation.
\end{enumerate}
\end{definition}
By \cite[Remark 1.18]{PhCr}, if $A$ is separably (nondegenerately) representable, then it is $\sigma$-finitely (nondegenerately) representable and by \cite[Proposition 1.25]{PhCr} if $A$ is separable then it is separably representable.

Note that it is not required that   representations of a unital algebra to be unital, but nondegenerate representations of a unital algebra are necessarily unital.

\vspace{2mm}
Let $A$ be a Banach algebra and let $\text{Aut(A)}$ denote the group of all continuous automorphisms of $A$. Let $G$ be a topological group,
by an \emph{action} of $G$ on $A$ we mean a homomorphism $g \mapsto {{\alpha} _ g}$  from $G$ to $\text{Aut(A)}$ such that
for any $a\in  A$, the map $g \mapsto \alpha _g(a)$ from $G$ to $A$ is continuous.  An action $\alpha$ is called  \emph{isometric}  if each
${\alpha} _g$
is.
If a topological group $G$ acts on  an $L^p$ operator algebra $A$, then
the triple
$(G, A, {\alpha})$
is called a
$G$-$L^p$
 \emph{ operator algebra}, and it is an \emph{isometric}
  $G$-$L^p$
  operator algebra whenever
$\alpha$
is isometric.

A $G$-$L^p$
  operator algebra $(G, A, \alpha)$ is said to be separable If $A$ is separable and it is said to be   nondegenerately representable, $\sigma$-finitely representable, whenever $A$ has the corresponding property in the sense of Definition \ref{1}.

 As an example, let $p\in [1,\infty]$, let $X$ be a locally compact Hausdorff space, and let $G$ be a locally
compact group which acts continuously on X, that is the corresponding action  $(g, x)\mapsto g \cdot x $ is a continuous map from $G\times X$ to $X$. Then $C_0(X)$ is an $L^p$ operator algebra
and the action $\alpha$ of $G$ on $C_0(X)$ defined by
$\alpha_{g}(f)(x)=f(g^{-1}x)$ for $f\in C_0(X),\; g\in G$ and $x\in X$,
makes
 $(G,C_0(X),\alpha)$
an isometric $G$-$L^p$ operator algebra, see \cite[Example 2.4]{PhCr}.

\vspace{2mm}
Let $(G,A,\alpha )$ be a $G$-$L^p$ operator algebra. A two-sided ideal $I$ of $A$ satisfying $\alpha_g(I)\subset I$ for all $g\in G$, is called a $G$-\emph{invariant} ideal.
We say that $A$ is $G$-\emph{simple} if $\{0\}$ and $A$ are the only  $G$-invariant two-sided ideals. We note that according to some authors, a Banach algebra is said to be simple (resp.\ $G$-simple) if it has no nontrivial \emph{closed} two-sided ideals (resp.\ $G$-invariant  ideals). For unital Banach algebras these two notions coincide.
\begin{remark}

 Let $A$ be a Banach algebra, let $G$ be a locally compact group with left Haar measure $\nu$, and let
$\alpha\colon G \to {\text{\text{Aut(A)}}}$
 be an action of $G$ on $A$. Then
$C_c(G,A,\alpha )$,   the vector space of all compactly supported continuous functions
from $G$ to $A$ is an associative algebra over  $\mathbb{C}$, when it is equipped with the twisted convolution product defined by
\begin{equation*}(ab)(g) = \int_{G}{a(h)\, \alpha_{h}(b(h^{-1}g))\, d\nu(h)}\end{equation*}
for $a, b \in C_c(G,A,\alpha )$ and $g \in G$.

\end{remark}

Let $p\in [1, \infty]$. Let $G$ be a topological group, and let $(G, A, \alpha)$ be a $G$-$L^p$ operator algebra. Take a measure space $(X, \mathcal B, \mu)$. A \emph{covariant
representation} of $(G, A, \alpha)$ on $L^p(X, \mu)$ is a pair
 $(\upsilon, \pi)$ consisting of a representation
$ g\mapsto{\upsilon_g}$ from $G$ to the group of invertible operators on  $L^p(X, \mu)$ such that $g\mapsto \upsilon_g{\xi}$ is continuous
for all $\xi \in L^p(X, \mu)$, and a representation
 $\pi\colon A \to B(L^p(X, \mu))$ such that for all $ g \in G$ and $a \in A$, we have
\begin{equation*}\pi(\alpha_g(a)) = \upsilon_g\pi(a)\upsilon_g^{-1}.
\end{equation*}

A covariant representation $(\upsilon, \pi)$ of $(G, A, \alpha)$ is \emph{contractive} (resp.\ \emph{isometric}) if   $\|\upsilon_g\|\leq 1$ for all
$g \in G$ and $\pi$ is contractive (resp.\ \emph{isometric}). It is
\emph{separable}, \emph{$\sigma$-finite}, or \emph{nondegenerate} whenever $\pi$ has the corresponding property.
 Note that, if $(\upsilon, \pi)$ is contractive then for each $g\in G$,  $\upsilon_g$ is an isometric bijection.

Let $p\in [1, \infty]$ and $A$ an $L^p$ operator algebra. If $G$ is a locally compact group with a left Haar measure $\nu$ then any  covariant representation $(\upsilon, \pi)$  of $(G, A, \alpha)$ on some $L^p(X, \mu)$ leads to a representation $\upsilon \ltimes
 \pi$ of  $C_c(G, A, \alpha)$ on $L^p(X, \mu)$ defined by
\begin{equation}\label{cv1} (\upsilon\ltimes \pi)(a)\xi= \int_{G}(\pi(a(g))\;\upsilon_{g} \xi\ d\nu(g)
\end{equation}
for $a\in C_c(G, A, \alpha)$ and $\xi \in L^p(X, \mu)$. This integral is defined by duality, that is for every $\omega$ in the dual space $  L^p(X, \mu)^*$ of $  L^p(X, \mu)$ we should have
$$\omega\big((\upsilon\ltimes \pi)(a)\xi\big)=\int_{G}\omega\big((\pi(a(g))\;\upsilon_{g} \xi\big)\ d\nu(g)$$

 Here we bring some parts of Lemma 2.11 of \cite{PhCr}.

\begin{lemma}\label{2.11} Let $p\in [1, \infty)$. Let $G$ be a locally compact group with left Haar measure $\nu$, and let $(G, A, \alpha)$ be an isometric $G$-$L^p$ operator algebra. Take a measure space
 $(X, \mathcal B, \mu)$ and let $\pi_0: A \rightarrow B(L^p(X, \mu))$ be a contractive representation. Then the following   hold.
\begin{enumerate}
\item[(i)] There exists a unique  representation $\upsilon \colon G \to B(L^p(G\times X,\nu \times \mu))$ such  that for all $g, h \in G, \; x\in X$ and $\xi \in L^p(G\times X,\nu \times \mu),$
$$\upsilon_g(\xi)(h, x)= \xi(g^{-1}h, x),$$
and this $\upsilon$ is isometric.
\item[(ii)] There exists a unique  representation $\pi \colon A  \to L^p(G\times X,\nu \times \mu)$ such that
for $a\in A\,, h\in G$ and $\xi \in C_c(G, L^p(X, \mu))\subset L^p(G\times X,\nu \times \mu)$ we have
\begin{equation}\label{cv}(\pi(a)\xi)(h)=\pi_0(\alpha_h^{-1}(a))(\xi(h)),\end{equation}
and this $\pi$ is contractive.\\
If $G$ is countable and discrete and $\nu$ is the counting measure, then according to the  identification of $L^p(G\times X,\nu \times \mu)$ with $l^p(G, L^p(X, \mu)$ as in  \cite[Remark  2.10]{PhCr},  (\ref{cv}) holds for all $\xi \in L^p(G\times X,\nu \times \mu)$.
\item[(iii)] The pair $ (\upsilon, \pi)$ is covariant. Further, if $\pi_0$
 is nondegenerate then so is $\pi$.
\item[(iv)] If  $G$ is second countable and $\mu$ is $\sigma$-finite, then $\nu \times \mu$ is $\sigma$-finite.
\item[(v)] If $G$ is second countable and $L^p(X, \mu)$ is separable, then $L^p(G \times X, \nu \times \mu)$
is separable.\end{enumerate}\end{lemma}
The covariant representation $(\upsilon, \pi)$ obtained as above
 is called the \emph{regular covariant representation of $(G,A, \alpha)$ associated
to $\pi_0$}. Any representation constructed in this way is called a \emph{regular contractive
covariant} representation. It is called \emph{separable}, \emph{$\sigma$-finite}, or \emph{nondegenerate} whenever the
 representation $\pi_0$ has the corresponding property.
 \vspace{2mm}

We now come to define  \textit{$L^p$ operator crossed products}. For  technical reasons as mentioned in \cite{PhCr}, $L^p$ operator crossed products are defined for second countable locally compact groups. To study the theory in a more general framework we refer to Section 3 of \cite{M1}.
\begin{definition}\label{crossed}(\cite[Definition 3.3]{PhCr})
Let $p\in[1,\infty)$, let $G$ be a second countable locally compact  group,
and let
$(G, A, \alpha)$
be an isometric
$G$-$L^p$
operator algebra which is nondegenerately $\sigma$-finitely representable. Following \cite[Definition 3.2]{M1} but with considering the family $\mathcal R$ to be as below, we define two corresponding crossed products.
\begin{enumerate}
\item [(i)] Let $\mathcal{R}^p$ be the family of all  covariant representations  coming from nondegenerate
$\sigma$-finite contractive representations of $A$. We define an algebra seminorm $\tau(\cdot)$ on $C_c(G, A, \alpha)$  by
$$ \tau(f) = \sup_{(\nu, \pi)\in {\mathcal{R}^p}} \| \nu \ltimes \pi (f)\| $$  for any $f\in C_{c}(G, A, \alpha)$. The \emph{ full $L^p$ operator crossed product}, $F^p(G, A, \alpha)$, is the completion of ${C_{c}(G, A, \alpha)}/ \ker(\tau)$ in the norm $\|\cdot\|$ induced by $\tau$.

\item [(ii)] Consider the family
$\mathcal{R}^p_{\mathrm{r}}$ of all regular covariant representations  coming from nondegenerate
$\sigma$-finite contractive representations and define an algebra seminorm $\tau_r(\cdot)$ on $C_c(G, A, \alpha)$ by
 $$ \tau_r (f) = \sup_{(\nu, \pi)\in \mathcal{R}^p_{\mathrm{r}}} \| \nu \ltimes \pi (f)\|, $$  for any $f\in C_{c}(G, A, \alpha)$. The completion of ${C_{c}(G, A, \alpha)}/ \ker(\tau_r)$ in the norm $\|\cdot\|_r$ induced by $\tau_r$ is called the \emph{ reduced $L^p$ operator crossed product} and is denoted by $F^p_{\mathrm{r}}(G, A, \alpha)$.
\end{enumerate} By \cite[Lemma 3.4]{PhCr},  $F^p(G, A, \alpha)$ and $F^p_{\mathrm{r}}(G, A, \alpha)$ exist as in \cite[Definition 3.2]{M1}.
 \end{definition}
 By assumption on  $(G, A, \alpha)$ in  the above definition,
 $A$ has a
 nondegenerate $\sigma$-finite isometric representation, see Definition \ref{1}(v).  Thus
    \cite[Proposition 3.11]{PhCr} implies that $\tau_r$ is actually a norm on ${C_{c}(G, A, \alpha)}$ and since $\tau_r \leq \tau$ the same is true for
   $\tau$. Moreover, \cite[Corollary 3.12]{PhCr} implies that ${C_{c}(G, A, \alpha)}$   can be considered as a subalgebra of both crossed products.

If we let $A=\mathbb C$ in Definition \ref{crossed}(i)-(ii), then the corresponding crossed products are called the $L^p$  operator group algebra and the reduced $L^p$ operator group algebra denoted by $F^p(G)$ and $F^{p}_r(G)$, respectively.
\begin{definition} \label{pw}(\cite[Definition, p.232]{PLH})
A  group $G$ is said to be a \emph{Powers group} if for any nonempty finite subset $F\subset G \setminus \{1\}$ and any integer $ k\geq 1$, there exist a disjoint partition $ G = C \amalg D$ and elements $ h_1, \ldots, h_k \in {G}$ such that
\begin{enumerate}
\item[(i)] $\,\, g C\cap C = \varnothing   \text{ for all}\,\,  g \in F,$
\item[(ii)] $ h_j D\cap h_l D = \varnothing \text{ for } j, l \in \{1, \ldots, k\}\, \text{with }\, j\neq l.$
\end{enumerate}
\end{definition}
\begin{remark}
The main idea of constructing Powers groups goes back to R. Powers and his proof of $C^*$-simplicity of nonabelian free groups \cite{Po}, for which he used combinatorial properties of free groups.
It is easily proved that free groups are Powers groups, see step two in the proof of \cite[Theorem 3]{H}.
\end{remark}
Some more examples of Powers groups are listed below.

\begin{enumerate}
\item[ (i)] Nonelementary torsion free  Gromov-hyperbolic groups; in particular, all nonabelian free groups \cite[Corollary 12]{H1}.
 \item[ (ii)] Certain amalgamated free products \cite[Proposition 10] {PLH}.
 \item[ (iii)] Nonsolvable subgroups of $PSL(2, \mathbb R)$ \cite[Proposition 5] {PLH}.
\item[ (iv)]   Any lattice $G$ in $ PSL(n, \mathbb C)$,  $n = 2, 3$ \cite[Proposition 13] {PLH}.
\end{enumerate}

Since 1985 when de la Harpe introduced Powers groups, many results have been obtained for these groups. Here we quote some of the more well-known ones.
 Powers groups are $C^*$-simple \cite[Proposition 3]{PLH},  they are also \emph{icc}  \cite[Proposition 1(a)]{PLH}. We recall that a  group $G$ is called an \emph{icc} group   if it is infinite  and if all its conjugacy classes distinct from $\{1\}$ are infinite. Furthermore, they are not amenable \cite [Proposition 1]{PLH}, and they do not even have  nontrivial amenable normal subgroup \cite[Proposition 1.6]{PS}.
For more details on the properties of Powers groups see \cite{PLH} and \cite{H1}.
\section{THE MAIN RESULTS}

In this section we present the main results regarding the simplicity and  characterizing  traces for reduced $L^p$ operator crossed products by countable Powers groups, for $p \in (1, \infty)$. We generalize the main results of de la Harpe and Skandalis in \cite{HS} in certain cases. To obtain these results we follow the main line of Powers \cite{Po}, but the proof of our key result, Theorem \ref{*}, is based on the duality between $L^p$ and $L^q$ spaces.
{\vspace {2mm}}

Throughout this section, we assume   that $A$ is a separable unital  $L^p$ operator algebra  on some  $\sigma$-finite measure space, that $p,q$ are conjugate exponents, and that $G$ is  a countable discrete group  with identity element $e$ and counting measure $\nu$. The unit element of $A$ is denoted by $1_A$ and the norm of $F^p_{\mathrm{r}} (G, A, \alpha)$  will be denoted by $\|\cdot\|_{\mathrm{r}}$.

{\vspace {2mm}}
  For $g\in G$,  let $u_g$ be the characteristic function of $\{g\}$ as a member of $C_c(g, A, \alpha)$. We may embed $G$  canonically into $C_c(G, A, \alpha)$ via the injective group homomorphism $g\mapsto u_g$ from $G$ to the group of invertible elements of $C_c(g, A, \alpha)$. It is easy to see that $u_e$ is the unit element of $C_c(g, A, \alpha)$ and ${u_{g}}^{-1}=u_{g^{-1}}$ for all $g\in G$.

Using \cite[Remark 4.6]{PhCr}, when it is necessary, we will identify $A$ as a subalgebra of $F^p_{\mathrm{r}} (G, A, \alpha)$ by considering the isometric homomorphism $a\mapsto au_e$. Note that $F^p_{\mathrm{r}} (G, A, \alpha)$ is a unital Banach algebra with unit element $1_A u_e$ which we will  identify it with $1_A$.
{\vspace {2mm}}

We begin with a lemma.
\begin{lemma}\label{11}
Let $p,\; q \in (1, \infty)$, let $k\in \mathbb{N}$ and let
$\lambda_1, \lambda_2, \ldots, \lambda _{k}, \gamma_{1}, \gamma_{2}, \ldots, \gamma_{k} \in \mathbb{R}$
be positive numbers such that  $\sum_{i=1}^{k}{\lambda_{i}^p}\leq  1$  and   $\sum_{i=1}^{k}{\gamma_{i}^q} \leq  1.$ Then
  $$\sum_{i=1}^{k}{\lambda_{i}}\leq {k^{\frac{1}{q}}}\hspace{2em} \text{and} \hspace{2em}  \sum_{i=1}^{k}{\lambda_{i} \gamma_{i}}\leq {1}. $$
\end{lemma}

\begin{proof}
It is immediate from H\"older's inequality.
\end{proof}

We need the following proposition in the proof of Theorem \ref{*}. Here for $a\in C_c(G, A, \alpha)$ and $g\in G$,  $a(g)$ will be denoted by $a_g$.
\begin{proposition} \label{CE}(\cite[Proposition 4.8, Proposition 4.9 (1)]{PhCr})
 Let $p\in [1,\infty)$, let $G$ be a countable discrete group, and let
$(G, A, \alpha)$ be a separable nondegenerately representable isometric $G$-$L^p$ operator
algebra. Then associated to each element $g\in G$, there is a
linear map $E_g \colon F^p_{\mathrm{r}} (G,A, \alpha) \to A $
with $\| E_g \| \leq 1$ such that if
$$a = \sum_{g\in G}{a_g u_g} \in C_c(G, A, \alpha)$$
then $E_g (a) = a_g$.
 Further, if $a \in F^p_{\mathrm{r}}(G, A, \alpha)$ with $E_g(a) = 0$  for each $g\in G$, then $a = 0$.
\end{proposition}
 Under the same assumptions as in Proposition \ref{CE},
the bounded linear map $E \colon F^p_{\mathrm{r}} (G, A, \alpha) \to A$
 defined by
$$E ( \sum_{g\in G}{a_g u_g})= a_e$$
for
$\sum_{g\in G}{a_g u_g} \in  C_c (G, A, \alpha)$, is called the \emph{standard conditional expectation} from
$F^p_{\mathrm{r}} (G, A, \alpha)$ to  $A$. See Section 4 of \cite{PhCr} for more on this concept.
\vspace{0.2mm}

 The next theorem  has a key role in the proof of the main results.

\begin{theorem} \label{*}
Let $p \in (1, \infty)$,
 let $G$ be a Powers group, and let $(G, A, \alpha)$ be a separable nondegenerately representable isometric $G$-$L^p$ operator algebra. Let $ a \in F_{\mathrm{r}}^p (G, A, \alpha)$, and let $\epsilon > 0 $. Then there exist $ k \in \mathbb{N}$  and $ h_1, h_2, \ldots, h_k \in G$ such that the
  averaging operator $ T \colon F^p_{\mathrm{r}} (G, A, \alpha) \to F^p_{\mathrm{r}} (G, A, \alpha),$
  defined by

 $$ T(b) = \frac{1}{k} \sum_{j=1}^{k}{u_{h_j} b u_{{h_j}}^{-1}},$$
 satisfies $\| T(a - E (a)) \|_{\mathrm{r}} < \epsilon.$
\end{theorem}

\begin{proof}
First we take $a \in C_c (G, A, \alpha)$ with $E (a) =0.$ That is,  there exist
 $n \in\mathbb{N} $, $g_1,\, g_2, \ldots,\, g_n \in G \setminus \{1\}$\,
   and nonzero elements $a_{g_1}, a_{g_2}, \ldots, a_{g_n} \in {A}$  such that $a = \sum_{i = 1}^{n}{a_{g_i} u_{g_i}}$.  Let $\|\cdot\|_1$ be the restriction of the $L^1$ norm to $C_c (G, A, \alpha)$. We may assume that $a\neq 0$. Choose $ k\in \mathbb{N}$ such that
$$ k^{-1}\;+\;k^{-\frac{1}{p}}\;+\; k^{-\frac{1}{q}} < \frac{\epsilon }{2 \|a\|_{1}}\;.$$
Put $F = \{ g_1, \ldots, g_n\}$. Since $G$ is a Powers group, for this $F$ and $k$ there exists a partition $\{C, D\}$ of $G$ and $h_1, \ldots h_k$ in $G$ which satisfy Definition  \ref{pw}.

Let $\pi_0$ be an arbitrary nondegenerate $\sigma$-finite contractive  representation of $A$ on $L^p(X, \mu)$ for some measure space $(X, \mathcal B, \mu)$ and let $(\upsilon, \pi)$ be the regular covariant representation associated  to $\pi_0$.

Let   $\chi_{_S}$ denote the characteristic function of
 $S\subset G \times X$. For each $j\in \{1 \ldots k\}$, define the idempotent operator
\begin{align*}
e_j&\colon L^p (G \times X, \nu \times \mu ) \,\,\to \,\,\, L^p (G \times X, \nu \times \mu)\\
 &\hspace*{2.83 cm}\xi \hspace*{.008cm} \mapsto \hspace*{.3cm} \chi_{_{(h_j D) \times X}}\cdot\,\xi
\end{align*}
then
\begin{align*}
e_j^* &\colon L^q (G \times X, \nu \times \mu) \,\, \to \,\,\, L^q (G \times X, \nu \times \mu )\\
 &\hspace*{2.78cm}\eta \hspace*{0.08cm} \mapsto \hspace*{.3cm} \chi_{_{(h_j D) \times X}}\cdot\,\eta
\end{align*}
is the adjoint operator of $e_j$.
%By Definition \ref{pw}(ii), for distinct $j, l \in \{ 1, 2, \ldots, k \}$  the functions in the range of $e_j$ and $e_{l}$ have disjoint supports, and the %same is true for the idempotents $e_j^*$ and $e_{l}^*$.

Let $ \xi \in {L^p (G \times X, \nu \times \mu )}$ and $\eta \in {L^q (G \times X, \nu \times \mu)} $
satisfy  $\|\xi\|_p = \|\eta\|_q=1,$
 so \begin{align*}\sum_{j = 1}^{k}{\|e_{{j}} \xi \|_p^p \leq \| \xi\|_p^p = 1} \,\,\,\, \text{and} \,\,\,\, \sum_{j = 1}^{k}{\|e_{{j}}^* \eta \|_q^q \leq \| \eta\|_q^q =1}.\end{align*}
Define the averaging operator $T \colon F^p_{\mathrm{r}} (G, A, \alpha) \to F^p_{\mathrm{r}} (G, A, \alpha)$ by
 \begin{eqnarray*} T(b) = \frac{1}{k} \sum_{j=1}^{k}{u_{h_j} b u_{{h_j}}^{-1}}\;.\end{eqnarray*}
Since by \cite[Lemma 4.13]{PhCr} for each $j=1, \ldots k$, left and right multiplication operators by $u_{h_j}$ are isometric maps  on $F^p_{\mathrm{r}} (G, A, \alpha)$, we have $\|T\| \leq 1$.
    Consider the representation $\upsilon \ltimes \pi$ of $C_c (G, A, \alpha)$ as given in Equation (\ref{cv1}). Note that for each $g\in G$, $(T(a))(g)= \frac{1}{k} \sum_{j = 1}^{k} \alpha_{h_j}(a_{h_{j}^{-1} g h_{j}})$. Therefore for the given $\xi \in L^p(G \times X, \nu\times \mu)$ we have
\begin{align*}
((\upsilon \ltimes \pi) T(a))\xi &= \sum_{g\in G}\pi \left(T(a)(g)\right)\upsilon_{g}\xi
\\
&=\sum_{g\in G}\pi\left(\frac{1}{k} \sum_{j = 1}^{k} \alpha_{h_j}(a_{h_{j}^{-1} g h_{j}})\right)\upsilon_g \xi\;.
\end{align*}
Since the support of $a$ is the set $\{g_1, \ldots, g_n\}$ it follows that
\begin{equation*}((\upsilon \ltimes \pi) T(a))\xi= \frac{1}{k}\sum_{j = 1}^{k} \sum_{i = 1}^{n}\pi \left(\alpha_{h_j}(a_{g_i})\right)\upsilon_{h_j g_i h_{j}^{-1}} \xi \;.\end{equation*}
 Using H\"older's inequality we then have
\begin{align*}
&\left|\left\langle{(\upsilon \ltimes \pi) T(a) \xi, \eta}\right\rangle \right|
\\
&= \left|\left\langle{ \frac{1}{k}\sum_{j = 1}^{k} \sum_{i = 1}^{n}(\pi (\alpha_{h_j}(a_{g_i}))\,\upsilon_{h_j g_i h_{j}^{-1}})\, \xi,\, \eta} \right\rangle\right|
\\
&= \left|\left\langle{ \frac{1}{k}\sum_{j = 1}^{k} \sum_{i = 1}^{n}(\pi (\alpha_{h_j}(a_{g_i}))\upsilon_{h_j g_i h_{j}^{-1}} (e_j + (1- e_j)) \xi , (e_j^*+(1-e_j^*))\eta} \right\rangle\right|
\\
&\leq \frac{1}{k}{\sum_{j = 1}^{k}{\sum_{i=1}^{n}{\left|{\left\langle{\pi(\alpha_{h_j} (a_{g_i})) \upsilon_{h_j g_i {h_j}^{-1}} }{(e_j + (1- e_j))\xi, (e_j^*+(1-e_j^*))\eta}\right\rangle}\right|}}}
\\
&\leq \frac{1}{k}\sum_{j=1}^{k}\sum_{i=1}^{n} \big(\left\|\pi(\alpha_{h_j} (a_{g_i}))\right\|\, \big(\| e_{j} \xi\|_p\cdot \| e_{j}^* \eta\|_q + \|(1-e_j)\xi\|_p \cdot \|e_j^* \eta\|_q
\\
& \hspace{.9em}{\mbox{}} + \|e_{j} \xi \|_p\cdot  \| (1-e_{j}^{*})\eta\|_q \big)
+ \left| \left\langle \pi (\alpha_{h_j}(a_{g_i}))\upsilon_{h_jg_ih_j^{-1}}(1-e_j)\xi,(1-e_j^*)\eta \right\rangle \right|\big).
\end{align*}
Now  applying Lemma \ref{2.11}(i)-(ii), for each $h\in G$ and $x \in X$ we have
 \begin{align*}
 &\left(\pi (\alpha_{h_j}(a_{g_i}))\upsilon_{h_jg_ih_j^{-1}}(1-e_j)\xi\right)(h, x)
 \\
 &\hspace{6em}=\left(\big(\pi (\alpha_{h_j}(a_{g_i}))\upsilon_{h_jg_ih_j^{-1}}(1-e_j)\xi\big)(h)\right)(x)
 \\
 &\hspace{6em}=\left(\pi_0(\alpha_{h^{-1}h_j}(a_{g_i}))(\upsilon_{h_jg_ih_j^{-1}}(1-e_j)\xi(h))\right)(x)
 \\
 &\hspace{6em}=\left(\pi_0(\alpha_{h^{-1}h_j}(a_{g_i}))(\chi_{_{(h_jD)^{c}\times X}} \cdot \xi)(h_jg_i^{-1}h_j^{-1}h)\right)(x).
 \end{align*}
 We can now exploit the properties of  the Powers group $G$, as follows. If $h \notin h_j D$ then $h_j^{-1}h\in C$. Therefore $g_i^{-1}h_j^{-1}h \notin C$ and hence this element belongs to $D$.  It follows that $h_jg_i^{-1}h_j^{-1}h\in h_jD$, so
 for all $h\notin h_jD$
 $$ \big(\pi (\alpha_{h_j}(a_{g_i}))(\upsilon_{h_jg_ih_j^{-1}}(1-e_j)\xi)\big)(h, x)=0,$$
and we arrive at
$$\left\langle \big(\pi (\alpha_{h_j}(a_{g_i}))\upsilon_{h_jg_ih_j^{-1}}(1-e_j)\big)\xi, (1-e_j^*)\eta \right\rangle=0.$$
  Therefore by Lemma \ref{11}
\begin{align*}
&\left|\left\langle{(\upsilon \ltimes \pi) T(a) \xi, \eta}\right\rangle \right|
\\
&\hspace{6em}\leq \frac{1}{k}\|a\|_{1}\sum_{j=1}^{k}  \left(\| e_{j} \xi\|_p \cdot \| e_{j}^* \eta\|_q + \|e_j^* \eta\|_q
+ \|e_{j} \xi \|_p\right)
\\
&\hspace{6em}\leq \frac{1}{k}\|a\|_{1}  \left(1 + k^{\frac{1}{p}}\;+ \; k^{\frac{1}{q}}\right)
\\
&\hspace{6em} = \|a\|_{1}( k^{-1} + k^{-\frac{1}{q}} + k^{-\frac{1}{p}}) < {\frac{\epsilon}{2}}\; .
%\\
% &\hspace{6em} = n \|a\|_{\mathrm{r}} \left(k^{-1} + k^{-\frac{1}{p}} + k^{-\frac{1}{q}}\right).
\end{align*}
Since $ \xi \in L^p (G\times X, \nu \times \mu) \,\, \mbox{and} \,\, \eta \in L^q (G \times X, \nu \times \mu)$ are arbitrary elements of norm 1,  it follows that the norm of  $(\upsilon \ltimes \pi)T(a)$, as an element of $B(L^p(X \times G, \mu \times \nu))$, does not exceed $\epsilon/2$. Since $(\upsilon, \pi)$ varies arbitrarily among all regular covariant representations  coming from nondegenerate $\sigma$-finite   contractive representations of $A$, it follows from the definition of the reduced crossed product norm that
$$\| T(a)\|_{\mathrm{r}} \leq \epsilon/2  < \epsilon\;.$$

Next, suppose that  $a\in C_c (G, A, \alpha)$ is arbitrary. Applying the previous step to the element $a - E (a)$, we may find an averaging operator $T$ such that
$$ \| T(a - E (a)) \|_{\mathrm{r}} < \epsilon \;.$$

 Finally, let $a \in F_{\mathrm{r}}^p (G, A, \alpha)$. By density of $C_c (G, A, \alpha)$ in $F_{\mathrm{r}}^p (G, A, \alpha)$,
  there exists $b \in C_c (G, A, \alpha) $ such that $\| a - b \|_{\mathrm{r}} < \frac{\epsilon}{3}$. Now by applying the first step, we may find an averaging operator $T$ such that
  $$ \| T(b - E(b)) \|_{\mathrm{r}} < \frac{\epsilon}{3}\;.$$
   % $$ \left\|\frac{1}{K} \sum_{k=1}^{K}{u_{h_k} b u_{{h_k}}^{-1}}  \right\| < \epsilon.$$
 Since $\| T \| \leq 1$ and $\| E \| \leq 1$, we then have
  $$\| T(a - E (a)) \|_{\mathrm{r}} \leq \| T(a) - T(b)\|_{\mathrm{r}} + \| T(b - E(b))\|_{\mathrm{r}} + \|T(E(b)) - T(E(a)) \|_{\mathrm{r}} <\epsilon\;.$$
 This completes the proof.
\end{proof}
We recall that a \emph{(normalized) trace} on a unital Banach algebra $A$ is a  bounded linear functional $\tau$ on $A$ (of norm 1  satisfying $\tau(1) = 1$) such that $\tau(ab) = \tau
(ba)$ for all $a, b \in A$.
Normalized traces on a unital  $C^*$-algebra  are exactly the tracial
states.
\begin{definition}Let $p \in [1, \infty]$, and  let $(G, A, \alpha)$ be a $G$-$L^p$ operator algebra. A  (normalized) trace on $A$ is said to be \textit{$G$-invariant} if it satisfies
$\tau(\alpha_g(a)) = \tau(a)$ for all $a \in A$.
\end{definition}
In the next step we are going to characterize  traces on $F^p_{\mathrm r}(G, A, \alpha)$. Note that  $\sigma \circ E$ is always a trace on  the reduced crossed product whenever $\sigma$ is a $G$-invariant trace on $A$. We will show that this is the only possible form for a trace on $F^p_{\mathrm r}(G, A, \alpha)$.
\begin{theorem}\label{trace}
Let $p \in (1, \infty)$, let $G$ be a countable Powers group, and let $(G, A, \alpha)$ be a separable nondegenerately representable isometric $G$-$L^p$ operator algebra. Then  each (normalized) trace  of $F^p_{\mathrm{r}} (G, A, \alpha)$ is of the form $\sigma\circ E$, where $\sigma$ is a $G$-invariant (normalized) trace on $A$. In particular, if $A$ has a unique normalized trace then so does $F^p_{\mathrm{r}} (G, A, \alpha)$.

\end{theorem}

\begin{proof}
Let $\tau$ be a trace on $F^p_{\mathrm{r}}(G, A, \alpha)$, let $a \in F^p_{\mathrm{r}}(G, A, \alpha)$, and let $\epsilon > 0$  be given. By Theorem \ref{*} there exist $k \in \mathbb{N}$ and $h_1, h_2, \ldots, h_k\in G$ such that
$$\left\| \frac{1}{k} \sum_{j = 1}^{k}{u_{h_j} (a -  E(a))u_{h_j}^{-1}}\right\|_{\mathrm{r}}< \epsilon.$$
By the tracial property of $\tau$, we then  have
$$ \left|\tau(a) - \tau(E(a)) \right|= |\tau(a- E(a))|=|\tau(\frac{1}{k} \sum_{j = 1}^{k}{u_{h_j} (a -  E(a))u_{h_j}^{-1}})| \leq \|\tau\|\epsilon.$$
Hence $\tau (a - E(a)) = 0$. Put $\sigma = \tau |_{{A}}$\;, then
\begin{eqnarray*}\tau(a) = \tau (E(a)) = \tau|_{{A}} (E(a)) = \sigma \circ E(a),\end{eqnarray*}
and $\sigma$ is a $G$-invariant  trace on $A$.
\end{proof}

The next lemma inspired by \cite[Lemma 9]{HS} will help us to obtain a generalization of Powers' idea as a main result of this article.
\begin{lemma} \label{S}
Let $G$ be a countable discrete group, let $(G, A, \alpha)$ is a separable nondegenerately representable isometric $G$-$L^p$ operator algebra and let $A$ be unital and $G$-simple. If $I$ is a nonzero two-sided ideal of $F^p_{\mathrm{r}} (G, A, \alpha)$, then there exists a nonzero element $a \in I$ such that $E(a) = 1_A$.
\end{lemma}

\begin{proof}
 First we show that  there is an element $b\in I$ with $E(b) \neq 0$. To this end, consider a nonzero element $c\in I$. By Proposition \ref{CE}, there exists
 $g \in G$  such that $E_g (c) \neq 0 $.  Since $C_c (G, A, \alpha) $ is dense in $F^p_{\mathrm{r}} (G, A, \alpha)$ we may choose a sequence $\{c_n\} \subset C_c (G, A, \alpha) $ such that $\lim_n c_n = c.$
Continuity of $E_g$ implies that  $\lim_n E_g (c_n) = E_g (c)$. On the other hand, $E_g (c_n) = E (c_n u_{g^{-1}})$
and thus $$ E(c u_{g^{-1}}) = \lim_n E(c_n u_{g^{-1}}) = \lim_n E_g (c_n) = E_g (c).$$
Clearly $c u_{g^{-1}} \in I$. So for $b = c u_{g^{-1}} \in I$  we have $E(b) \neq 0$. Define $J$ to be the two-sided ideal of $A$ generated by $\{\alpha_g(E(b)): g\in G\}$, $G$-Simplicity of $A$ implies that  $J = A$. Hence there are $ m\in \mathbb{N},\; g_1, \ldots, g_m \in G \,\, \text{and}\,\, a_1, \ldots, a_m, b_1, \ldots, b_m \in A$ such that
$$\sum_{i = 1}^{m}{a_i \alpha_{g_{i}} (E(b))} b_i  = 1_A.$$
 Take $a = \sum_{i = 1}^{m}{a_i u_{g_i} b \,  u_{g_i^{-1}} b_i} \in I. $
It is easy to see that $$E(a) = \sum_{i = 1}^{m}{a_i \alpha_{g_i} (E(b))} b_i = 1_A\; ,  $$
and we are done.
\end{proof}

%Now we are ready to prove the main result of this paper, that is a sufficient condition for simplicity of $F^p_{\mathrm{r}}(G,A, \alpha)$.

\begin{theorem} \label{main}
Let $p \in (1, \infty)$, let $G$ be a countable Powers group, and  let $(G, A, \alpha)$ be a separable nondegenerately representable isometric $G$-$L^p$ operator algebra such that $A$ is unital and $G$-simple, then $F^p_{\mathrm{r}} (G, A, \alpha)$ is simple.
\end{theorem}

\begin{proof}
Let $I$ be a nonzero two-sided ideal in $F^p_{\mathrm{r}} (G, A, \alpha)$. By Lemma \ref{S} there exists  $a \in I$  such that $ E(a) = 1$.
Applying Lemma \ref{*} to $a - E (a)$ and $\epsilon = \frac{1}{2}$ shows that
there exist $k \in {\mathbb N}$ and $h_1, \ldots, h_k \in G$ such that
$$ \left\|\frac{1}{k}\sum_{j=1}^{k}{u_{h_j} a u_{h_j}^{-1}- 1_A}\right\|_{\mathrm{r}}  =\left\|\frac{1}{k} \sum_{j = 1}^{k}{u_{h_j} (a - E (a)) u_{h_j}^{-1}}\right\|_{\mathrm{r}} < {\frac{1}{2}}\;.$$
Consequently, $I$ contains an invertible element $\frac{1}{k} \sum_ {j =1}^{k}{u_{h_j} a u_{h_j}^{-1}}.$ Therefore $I = F_{\mathrm{r}}^p (G, A, \alpha).$
This shows that $F_{\mathrm{r}}^p (G, A, \alpha)$ is simple.
\end{proof}
As a consequence, for $p\in (1, \infty)$, $F^{p}_{\mathrm r}(G)$,
 the reduced $L^p$ operator group algebra  of a countable Powers group $G$  is simple with a unique normalized trace. This generalizes \cite[Theorem 1.1]{PS} in the case of countable groups and \cite{Po} in $L^p$ level.

%Our next remark is a justification for nonsimplicity of $L^1$ operator group algebras, see \cite[Proposition 3.14]{PhCr}. We give a proof for the sake of convenience.
\begin{remark}
Let $G$ be a countable discrete group. By \cite[Proposition 3.14]{PhCr} for $p = 1$, the Banach algebras $l^1(G)$, $F^1_{\mathrm{r}}(G)$ and $F^1(G)$ are isometrically isomorphic.
Take the trivial homomorphism $\phi \colon G \to \mathbb{C}$. We then get an induced homomorphism $\tilde{\phi} \colon l^1(G) \to \mathbb{C}$ whose kernel is a nontrivial tow-sided ideal. As a result, non of the reduced and full $L^1$  operator group algebras of a countable discrete group is  simple.
\end{remark}

\begin{remark}
The hypothesis that $A$ is unital is essential in Theorem \ref{main}, indeed the example mentioned in the last part of \cite{HS} shows that  Theorem \ref{main} does not hold for  nonunital $L^p$ operator algebras even for $p=2$.
\end{remark}

%\begin{Acknowledgments*}
We would like to express our  thanks to  N. Christopher  Phillips for valuable comments and conversations. 
%\end{Acknowledgments*}

% We give three examples on how to type the bibliographical items, for
% articles, books, and dissertations.

\end{document}